\documentclass[11pt]{amsart}
\usepackage{amsfonts} % za izpis oznak za mnozice stevil
\usepackage{amsmath} % za definicijo matematicnih operatorjev, kot je npr. arctg, in nekaterih simbolov
\usepackage{amsthm} % izpise piko za stevilko definicije, ce je definicija vpeljana kot newtheorem
\usepackage{amssymb} % za razne simbole
\usepackage{enumitem} % da lahko nadaljujes enumerate
\usepackage{mathrsfs} % zakrivljena pisava

\DeclareMathOperator{\ann}{ann}

\newtheorem{theorem}{Theorem}[section]
\newtheorem{proposition}[theorem]{Proposition}
\newtheorem{lemma}[theorem]{Lemma}
\newtheorem{corollary}[theorem]{Corollary}

\theoremstyle{definition}
\newtheorem{definition}[theorem]{Definition}
\newtheorem{example}[theorem]{Example}

\newtheorem{question}[theorem]{Question}

\begin{document}

\title[Categorial properties of compressed zero-divisor graphs]{Categorial properties of compressed zero-divisor graphs of finite commutative rings}

\author{Alen \DJ uri\' c}
\address[A. \DJ uri\' c]{Faculty of Natural Sciences and Mathematics, University of Banja Luka, Mladena Stojanovi\' ca 2, 78000 Banja Luka, Bosnia and Herzegovina}
\email{alen.djuric@protonmail.com}

\author{Sara Jev\dj eni\' c}
\address[S. Jev\dj eni\' c]{Faculty of Natural Sciences and Mathematics, University of Banja Luka, Mladena Stojanovi\' ca 2, 78000 Banja Luka, Bosnia and Herzegovina}
\email{sarajevdjenic9@gmail.com}

\author{Nik Stopar}
\address[N. Stopar]{Faculty of Electrical Engineering, University of Ljubljana, Tr\v za\v ska cesta 25, 1000 Ljubljana, Slovenia}
\email{nik.stopar@fe.uni-lj.si}

\thanks{This research was supported by the Slovenian Research Agency, project number BI-BA/16-17-025.}

\begin{abstract}
We define a compressed zero-divisor graph $\varTheta(K)$ of a finite commutative unital ring $K$, where the compression is performed by means of the associatedness relation. We prove that this is the best possible compression which induces a functor $\varTheta$, and that this functor preserves categorial products (in both directions). We use the structure of $\varTheta(K)$ to characterize important classes of finite commutative unital rings, such as local rings and principal ideal rings.

%We define a compressed zero-divisor graph $\varTheta(K)$ of a finite commutative unital ring $K$ and show that this extends to a product preserving functor. The compression in $\varTheta(K)$ is performed by means of the associatedness relation, which is a refinement of the relation used in the definition of the compressed zero-divisor graph $\Gamma_E(K)$ introduced by Mulay. We show that the structure of $\varTheta(K)$ can be used to characterize important classes of finite commutative rings, such as local rings and principal ideal rings.
\end{abstract}

\maketitle

{\footnotesize \emph{Key Words:} compressed zero-divisor graph, categorial product, local ring, principal ideal ring


}

\section{Introduction}\label{intro}

The aim of this paper is to study the zero-divisor graphs of finite commutative rings with special attention devoted to categorial properties.
The zero-divisor graph of a commutative ring was first introduced by Beck \cite{Bec}, to investigate the structure of commutative rings.
For a given commutative ring $K$, Beck's zero-divisor graph $G(K)$ is a simple graph with vertex set $K$, such that two distinct vertices $a$ and $b$ are adjacent if and only if $ab=0$. Beck was mainly interested in the chromatic number and the clique number of the graph.
Later Anderson and Livingston \cite{And-Liv} defined a simplified version $\Gamma(K)$ of Beck's zero-divisor graph by including only nonzero zero-divisors of $K$ in the vertex set and leaving the definition of edges the same.
In particular, this graph is still a simple graph, but may have far fewer vertices in general. Their motivation for this simplification was to better capture the essence of the zero-divisor structure of the ring.
Several properties of $\Gamma(K)$ have been investigated, such as connectedness, diameter, girth, chromatic number, etc. \cite{And-Liv, Akb-Moh}.
In addition, the isomorphism problem for such graphs has been solved for finite reduced rings \cite{And-etal}.
Several authors have also investigated rings $K$ whose graph $\Gamma(K)$ belongs to a certain family of graphs, such as star graphs \cite{And-Liv}, complete graphs \cite{Akb-Moh}, complete $r$-partite graphs and planar graphs \cite{Akb-Mai-Yas,Smi}.
Similar type of zero-divisor graphs have been considered in other algebraic structures as well, namely, semirings and semigroups \cite{atani08,atani09,dem02,demdem,dem10,dem10rm,Obl}.

Although smaller, the graph $\Gamma(K)$ may still have very large set of vertices and edges.
To further reduce the size of the graph, Mulay \cite{Mul} introduced the graph of equivalence classes of zero-divisors $\Gamma_E(K)$, which was later called compressed zero-divisor graph by Anderson and LaGrange \cite{And-LaG}.
Two elements $a$ and $b$ of a commutative unital ring $K$ are equivalent if $\ann_K a=\ann_K b$.
The vertex set of $\Gamma_E(K)$ is the set of all equivalence classes of nonzero zero-divisors of $K$ and two distinct equivalence classes $[r]$ and $[s]$ are adjacent if and only if $rs=0$.
Compressed zero-divisor graphs were investigated in more details by Spiroff and Wickham \cite{Spi-Wic}, Coykendall, Sather-Wagstaff, Sheppardson and Spiroff \cite{Coy-etal} and Anderson and LaGrange \cite{And-LaG, And-LaG-2}. They considered similar graph properties that were previously considered for the zero-divisor graph.
The main advantage of the compressed zero-divisor graph $\Gamma_E(K)$ over the noncompressed graph $\Gamma(K)$ is that it can be relatively small even if the ring itself is large.
In particular, $\Gamma_E(K)$ can be a finite graph even if $K$ is an infinite ring and $\Gamma(K)$ an infinite graph. Nevertheless, graph $\Gamma_E(K)$ still captures the essence of the zero-divisor structure of the ring, since the elements that are identified by the above equivalence have the same neighbourhood in $\Gamma(K)$.

In this paper we will be dealing with a type of compressed zero-divisor graph of finite commutative unital rings. Our main focus will be to investigate the categorial properties of such graphs.
Our compressed zero-divisor graph, denoted by $\varTheta(K)$ (see Definition~\ref{def:theta}), is essentially a compression of Beck's original zero-divisor graph $G(K)$, except that we allow loops in the graph.
Unlike in the definition of $\Gamma_E(K)$, here, the compression is performed by means of the associatedness relation (recall that $a,b \in K$ are associated if $a=bu$ for some invertible element $u \in K$).
We remark that the associatedness relation is a refinement of the relation used in the definition of $\Gamma_E(K)$, hence, the graph $\Gamma_E(K)$ can easily be obtained from $\varTheta(K)$ by simply identifying the vertices of $\varTheta(K)$ with the same neighbourhood, and eliminating those vertices that do not correspond to zero-divisors.

The advantage of $\varTheta(K)$ over $\Gamma_E(K)$ is that it can be extended in a natural way to a functor from the category of commutative rings to the category of graphs.
The main reason why $\Gamma_E(K)$ does not extend to a functor in a natural way is that the corresponding equivalence relation induced by annihilator ideals is too coarse, it compresses the zero-divisor graph too much.
In fact, we show that in the class of finite unital rings the associatedness relation is the coarsest equivalence relation that still induces a functor $\varTheta$ (see Propositions~\ref{prop:best} and \ref{prop:functor} for details). Graph $\varTheta(K)$ is thus the best possible candidate for a categorial approach to compressed zero-divisor graphs of finite rings.

It turns out that functor $\varTheta$ has several favourable properties that connect the ring structure of $K$ and the graph structure of $\varTheta(K)$. In particular, it preserves categorial products, not only in the forward direction but, in some sense, also in the backward direction - a decomposition of graph $\varTheta(K)$ induces a decomposition of ring $K$ (see Theorem~\ref{prop:preproduct}). In the class of finite commutative rings this reduces the problem to local rings.
In addition, our main results, Theorems~\ref{prop:local--iff} and \ref{thm:pir-iff}, show that the structure of $\varTheta(K)$ can be used to characterize important families of rings within the category of finite commutative unital rings, namely, local rings and principal ideal rings.

\section{Preliminaries} 

Throughout the paper $K$ will be a finite commutative unital ring with unity $1$, unless specified otherwise. In particular, we consider the zero ring to be unital. We denote by $\sim$ the \emph{associatedness} relation on the set of elements of $K$. By definition $a\sim b$ if and only if $a=bu$ for some invertible element $u \in K$.
The associatedness class of an element $a \in K$, i.e. the equivalence class of $a$ with respect to $\sim$, will be denoted by $[a]$. The equivalence class with respect to any other equivalence relation $\approx$ will be denoted by $[\phantom{a}]_\approx$.
Recall that in a finite commutative unital ring $K$ every element is either a unit or a zero-divisor. Indeed, if $a \in K$ is not a zero-divisor, then the map $x \mapsto ax$ is injective and hence surjective, which means that $a$ is invertible.
The ring of integers modulo $m$ will be denoted by $\mathbb{Z}_m$. 

Let $G$ be an arbitrary, possibly non-simple, graph and $v$ a vertex in $G$. The \emph{neighbourhood} of $v$, i.e. the set of all vertices adjacent to $v$ (including possibly $v$), will be denoted by $N(v)$.
The graphs we will be dealing with will have no multiple edges and no multiple loops.
We will adopt the convention that a loop on vertex $v$ contributes $1$ to the degree of $v$, denoted $\deg (v)$. With this convention our graphs will satisfy $\deg (v)=|N(v)|$.

\section{Definition and categorial properties of $\varTheta(K)$}

It is easily verified that finite commutative unital rings form
a category with arrows being ring homomorphisms that preserve the identity element.
We will denote this category by $\mathbf{FinCRing}$. The category of undirected graphs and graph morphisms will be denoted by $\mathbf{Graph}$. Given a category $\mathbf{C}$, we will denote the class of objects of $\mathbf{C}$ by $\mathrm{obj}\mathbf{C}$. For $X,Y \in \mathrm{obj}\mathbf{C}$, the set of morphisms from $X$ to $Y$ will be denoted by $\mathbf{C}(X,Y)$.

In this paper we will take a categorial approach to zero-divisor graphs. We will focus on compressed zero-divisor graphs since these are usually much smaller then the standard zero-divisor graphs.
As mentioned in the introduction, Mulay's compressed zero-divisor graph $\Gamma_E(K)$ is not a good candidate for a categorial approach, because the formation of $\Gamma_E(K)$ does not extend to a functor $\mathbf{FinCRing} \to \mathbf{Graph}$ in a natural way. The problem is that graph $\Gamma_E(K)$ is compressed too much. Hence, our definition of zero-divisor graph will be different.
We want to compress the zero-divisor graph as much as possible, in such a way, that it will still induce a functor. The following proposition (along with Proposition~\ref{prop:functor}) essentially states that the associatedness relation is the best equivalence relation to do this.

\begin{proposition}\label{prop:best}
For each $K\in\mathrm{obj}\mathbf{FinCRing}$, let $\approx_{K}$
be an equivalence relation on $K$, such that the family $\left\{ \approx_{K}\right\} _{K\in\mathrm{obj}\mathbf{FinCRing}}$
induces a well defined functor $F:\mathbf{FinCRing}\to\mathbf{Graph}$
in the following way.
\begin{enumerate}
\item\label{enu:compression} For $K\in\mathrm{obj}\mathbf{FinCRing}$, the
vertices of $F\left(K\right)$ are equivalence classes of $\approx_{K}$, and
there is an edge between vertices $\left[a\right]_{\approx_{K}}$
and $\left[b\right]_{\approx_{K}}$ if and only if $ab=0$.
\item\label{enu:arrow} For $f\in\mathbf{FinCRing}$$\left(K,L\right)$,
we have $F\left(f\right)\left(\left[a\right]_{\approx_{K}}\right)=\left[f\left(a\right)\right]_{\approx_{L}}$.
\end{enumerate}
Then, for every $K\in\mathrm{obj}\mathbf{FinCRing}$, $a\approx_{K}b$
implies $a\sim b$.
\end{proposition}

\begin{proof}
Suppose $x\approx_{K}0$. By \ref{enu:compression}, there is an edge joining
$\left[1\right]_{\approx_{K}}$ and $\left[0\right]_{\approx_{K}}=\left[x\right]_{\approx_{K}}$.
Since edges have to be well defined, we deduce $1\cdot x=0$.
This shows that $\left[0\right]_{\approx_{K}}=\left\{ 0\right\} $
for any $K$.

Now, suppose $a\approx_{K}b$ and let $q:K \to K/aK$ be the canonical
projection. Then, by \ref{enu:arrow}, 
\[
\left[q\left(b\right)\right]_{\approx_{K/aK}}=F\left(q\right)\left(\left[b\right]_{\approx_{K}}\right)=F\left(q\right)\left(\left[a\right]_{\approx_{K}}\right)=\left[q\left(a\right)\right]_{\approx_{K/aK}}=\left[0\right]_{\approx_{K/aK}}.
\]
Thus, the above implies $q\left(b\right)=0$, hence $b\in aK$. Similarly,
$a\in bK$, so $aK=bK$. As remarked by Kaplansky in \cite[\S 2]{Kap}, in any artinian ring and, in particular, in any finite ring, this implies $a\sim b$.
\end{proof}

The above result thus motivates us to define compressed zero-divisor graphs in the
following way.

\begin{definition}\label{def:theta}
For a finite commutative unital ring $K$, $\varTheta\left(K\right)$
is a graph whose vertices are associatedness classes (including $\left[0\right]$
and $\left[1\right]$) of elements of $K$ and vertices $\left[u\right]$ and $\left[v\right]$ (not necessarily
distinct) are adjacent if and only if $uv=0$.
\end{definition}

Observe that the edges of graph $\varTheta\left(K\right)$ are well-defined.
In addition, the associatedness classes form a monoid under the well-defined multiplication $[x]\cdot[y]=[xy]$.

We remark that class $[0]$ contains only $0$ and class $[1]$ consists of all the units of the ring. We need to keep these two classes in the graph and also allow loops because we need them in order to obtain a functor. Every other class is represented by a nonzero zero-divisor, because in a finite ring every element is either a zero-divisor or a unit.

\begin{proposition}\label{prop:functor}
The mapping $K\mapsto\varTheta\left(K\right)$ extends to
a functor $\varTheta:\mathbf{FinCRing}\to\mathbf{Graph}$.
\end{proposition}

\begin{proof}
Let $f:K \to L$ be a unital ring homomorphism, where $K$
and $L$ are finite commutative unital rings. Define $\varTheta\left(f\right):\varTheta\left(K\right)\to\varTheta\left(L\right)$
by $\varTheta\left(f\right)\left(\left[x\right]\right)=\left[f\left(x\right)\right]$.
Observe that $\varTheta\left(f\right)$ is well-defined since $f$
preserves units, and clearly, $\varTheta\left(f\right)$ is a graph
homomorphism. In addition, $\varTheta\left(\textrm{id}_{K}\right)=\textrm{id}_{\varTheta\left(K\right)}$
and $\varTheta\left(g\circ f\right)=\varTheta\left(g\right)\circ\varTheta\left(f\right)$
for all morphisms $f:K \to L$ and $g:L \to M$. So $\varTheta:\mathbf{FinCRing}\to\mathbf{Graph}$
is a functor.
\end{proof}

Observe that both categories involved have all finite products.
Binary product in the category $\mathbf{FinCRing}$ is the direct product of rings, while binary product in category $\mathbf{Graph}$ is the tensor product of graphs (also called categorical product or Kronecker product). Recall that for graphs $G$ and $H$, their tensor product $G\times H$ is defined
as follows. The set of vertices of $G\times H$ is the Cartesian product
$V\left(G\right)\times V\left(H\right)$ and a vertex $\left(g,h\right)$
is adjacent to a vertex $\left(g',h'\right)$ if and only if both
$g$ is adjacent to $g'$ and $h$ is adjacent to $h'$. Final object in in the
category $\mathbf{FinCRing}$ is the zero ring $0$ and final
object in the category $\mathbf{Graph}$ is the graph with precisely
one vertex and one loop.

\begin{proposition}\label{prop:preserves--products}
The functor $\varTheta:\mathbf{FinCRing}\to\mathbf{Graph}$
preserves finite products.
\end{proposition}

\begin{proof}
It is sufficient to show that functor $\varTheta$ preserves binary products and
final object. Let $K,L\in\mathrm{obj}\mathbf{FinCRing}$. Since operations
in $K\times L$ are defined coordinate-wise, we have that $\left[\left(x,y\right)\right]=\left[\left(x',y'\right)\right]$
in $K\times L$ if and only if both $\left[x\right]=\left[x'\right]$
and $\left[y\right]=\left[y'\right]$. This shows that $V\left(\varTheta\left(K\times L\right)\right)=V\left(\varTheta\left(K\right)\right)\times V\left(\varTheta\left(L\right)\right)$.
In addition, $\left[\left(x,y\right)\right]$ is adjacent to $\left[\left(z,w\right)\right]$
in $\varTheta\left(K\times L\right)$ if and only if both $\left[x\right]$
is adjacent to $\left[z\right]$ in $\varTheta\left(K\right)$ and
$\left[y\right]$ is adjacent to $\left[w\right]$ in $\varTheta\left(L\right)$.
Hence, $\varTheta\left(K\times L\right)$ is isomorphic to the tensor
product of graphs $\varTheta\left(K\right)$ and $\varTheta\left(L\right)$
by the map $\left[\left(x,y\right)\right]\mapsto\left(\left[x\right],\left[y\right]\right)$.
Clearly, $\varTheta\left(0\right)$, the graph of the zero ring, is
the graph with precisely one vertex and one loop.
\end{proof}

As the following example shows, functor $\varTheta:\mathbf{FinCRing}\to\mathbf{Graph}$
does not preserve finite limits in general, hence it has no left adjoint
functor.

\begin{example}
Let $K=\mathbb{Z}_{2}\left[x\right]/\left(x^{3}\right)$ and let $f:K \to K$
be a unital ring homomorphism which maps $x$ to $x+x^{2}$. It is
easily verifed that the equalizer of $\textrm{id}_{K}$ and $f$ (i.e. the limit cone over the diagram $K \mathrel{\substack{\xrightarrow[\phantom{\textrm{id}_K}]{f}\\[-3.3ex]\xrightarrow[\textrm{id}_K]{\phantom{f}}}} K$), is
$E\xrightarrow{i}K$, where $E=\left\{ 0,1,x^{2},1+x^{2}\right\} $
and $i$ is an inclusion. Observe that 
\[
V\left(\varTheta\left(E\right)\right)=\left\{ \left\{ 0\right\} ,\left\{ x^{2}\right\} ,\left\{ 1,1+x^{2}\right\} \right\} .
\]
On the other hand, associatedness classes in $K$ form the set 
\[
V\left(\varTheta\left(K\right)\right)=\left\{ \left\{ 0\right\} ,\left\{ 1,1+x,1+x^{2},1+x+x^{2}\right\} ,\left\{ x,x+x^{2}\right\} ,\left\{ x^{2}\right\} \right\} .
\]
We see that $\varTheta\left(f\right)=\varTheta\left(\mathop{id}_{K}\right)=\textrm{id}_{\varTheta\left(K\right)}$,
so the equalizer of $\varTheta\left(\textrm{id}_{K}\right)$
and $\varTheta\left(f\right)$ is $\varTheta\left(K\right)\xrightarrow{\textrm{id}_{\varTheta\left(K\right)}}\varTheta\left(K\right)$.
Since
$\left|V\left(\varTheta\left(E\right)\right)\right|\neq\left|V\left(\varTheta\left(K\right)\right)\right|$,
we conclude that equalizer of $\textrm{id}_{K}$ and $f$ is not preserved
by functor $\varTheta:\mathbf{FinCRing}\to\mathbf{Graph}$.
\end{example}

The following theorem shows that, in a sense, the product is preserved in the reverse direction as well.

\begin{theorem}\label{prop:preproduct}
Suppose $K,L_{1},L_{2}\in\mathrm{obj}\mathbf{FinCRing}$
such that $\varTheta\left(K\right)\cong\varTheta\left(L_{1}\right)\times\varTheta\left(L_{2}\right)$.
Then $K=K_{1}\times K_{2}$ for some subrings $K_{1},K_{2}\subseteq K$
with $\varTheta\left(K_{1}\right)\cong\varTheta\left(L_{1}\right)$
and $\varTheta\left(K_{2}\right)\cong\varTheta\left(L_{2}\right)$.
\end{theorem}

\begin{proof}
If $\varTheta\left(L_{1}\right)\cong\varTheta\left(0\right)$,
then $\varTheta\left(L_{1}\right)\times\varTheta\left(L_{2}\right)\cong\varTheta\left(L_{2}\right)$
so we may take $K_{1}=0$ and $K_{2}=K$. We argue similarly if $\varTheta\left(L_{2}\right)\cong\varTheta\left(0\right)$.
So assume $\varTheta\left(L_{1}\right)\ncong\varTheta\left(0\right)$
and $\varTheta\left(L_{2}\right)\ncong\varTheta\left(0\right)$.

Let $f:\varTheta\left(L_{1}\right)\times\varTheta\left(L_{2}\right)\to\varTheta\left(K\right)$
be any isomorphism. Choose $k_{1},k_{2}\in K$ such that $f\left(\left(\left[1\right],\left[0\right]\right)\right)=\left[k_{1}\right]$
and $f\left(\left(\left[0\right],\left[1\right]\right)\right)=\left[k_{2}\right]$, and define
\begin{equation}
K_{1}=\ann\left(k_{2}\right) \quad\textup{and}\quad K_{2}=\ann\left(k_{1}\right).\label{eq:24comm}
\end{equation}
Clearly, $K_{1}$ and $K_{2}$ are ideals of $K$. If $x\in K_{1}\cap K_{2}$,
then 
\begin{align*}
\left[x\right]\in N\left(\left[k_{1}\right]\right)\cap N\left(\left[k_{2}\right]\right)
&=f\left(N\left(\left(\left[1\right],\left[0\right]\right)\right)\cap N\left(\left(\left[0\right],\left[1\right]\right)\right)\right)=\\
&=f\left(\left\{ \left(\left[0\right],\left[0\right]\right)\right\} \right)=\{[0]\}.
\end{align*}
Thus, $K_{1}\cap K_{2}=0$.

Note that the subgraph of $\varTheta\left(L_{1}\right)\times\varTheta\left(L_{2}\right)$,
induced by $N\left(\left(\left[0\right],\left[1\right]\right)\right)$,
is isomorphic to $\varTheta\left(L_{1}\right)$. Hence, the
subgraph $G_{1}$ of $\varTheta\left(K\right)$, induced by $N\left(\left[k_{2}\right]\right)$,
is also isomorphic to $\varTheta\left(L_{1}\right)$. Clearly,
$V\left(G_{1}\right)=\left\{ \left[x\right]\in V\left(\varTheta\left(K\right)\right):x\in K_{1}\right\} $
and $k_{1}\in K_{1}$. Since $K_{1}$ is an ideal, we thus have $\left[k_{1}^{2}\right]\in V\left(G_{1}\right)$. Observe that $[k_1^2] \neq [0]$, since $[k_1]$ has no loop due to the fact that $L_1 \neq 0$.
Suppose $\left[k_{1}^{2}\right]\neq\left[k_{1}\right]$. Then $\left[k_{1}^{2}\right]\in V\left(G_{1}\right)\setminus\left\{ \left[0\right],\left[k_{1}\right]\right\} $.
Since $G_{1}\cong\varTheta\left(L_{1}\right)$, there is only one vertex in $G_1$
that is adjacent to every vertex in $G_1$, i.e. $\left[0\right]$, and there is only one vertex in $G_{1}$
whose only neighbour in $G_{1}$ is $\left[0\right]$, i.e. $f\left(\left(\left[1\right],\left[0\right]\right)\right)=\left[k_{1}\right]$.
This implies that $\left[k_{1}^{2}\right]$ has a neighbour in
$G_{1}$ different from $\left[0\right]$, say $\left[a\right]$,
where $a\in K_{1}$. Hence, $k_{1}^{2}a=0$ because $G_{1}$ is an
induced subgraph of $\varTheta\left(K\right)$. This imples that $\left[k_{1}a\right]$
is adjacent to $\left[k_{1}\right]$ in $\varTheta\left(K\right)$,
and since $K_{1}$ is an ideal, $\left[k_{1}a\right]\in V\left(G_{1}\right)$.
Therefore, $k_{1}a=0$ because, by the above, [0] is the only neighbour of $[k_1]$ in $G_1$. Similarly, this implies that
$\left[a\right]$ is adjacent to $\left[k_{1}\right]$, hence
$a=0$, a contradiction. We have thus shown that $\left[k_{1}^{2}\right]=\left[k_{1}\right]$.
In particular, $k_{1}=k_{1}^{2}u_{1}$ for some unit $u_{1}\in K$.

Observe that $k_{1}\left(1-k_{1}u_{1}\right)=0$, hence $1-k_{1}u_{1}\in K_{2}$
by (\ref{eq:24comm}). If $1-k_{1}u_{1}=0$, then $k_{1}$ is a unit in
$K$, hence $\left[k_{1}\right]=\left[1\right]$. But this would imply
that $\left[0\right]$ is the only neighbour of $\left[k_{1}\right]$
in $\varTheta\left(K\right)$, which would further imply
$K_{2}=0$. In this case, $\varTheta\left(L_{2}\right)\cong\varTheta\left(0\right)$,
a contradiction. So $1-k_{1}u_{1}\neq0$.

Suppose $\left[1-k_{1}u_{1}\right]\neq\left[k_{2}\right]$. Let $G_{2}$
be the subgraph of $\varTheta\left(K\right)$, induced by $N\left(\left[k_{1}\right]\right)$.
Then the same argument as above shows that $\left[1-k_{1}u_{1}\right]\in V\left(G_{2}\right)\setminus\left\{ \left[0\right],\left[k_{2}\right]\right\} $
has a neighbour in $G_{2}$ different from $\left[0\right]$,
say $\left[b\right]$, where $0 \neq b\in K_{2}$. Hence, 
\begin{equation}
\left(1-k_{1}u_{1}\right)b=0\label{eq:24-1comm}
\end{equation}
 because $G_{2}$ is an induced subgraph of $\varTheta\left(K\right)$.
Since $k_{2}k_{1}=0$, we have $k_{2}=k_{2}\left(1-k_{1}u_{1}\right)$.
Hence, $k_{2}b=0$ by (\ref{eq:24-1comm}). This implies $b\in K_{1}$,
so $b\in K_{1}\cap K_{2}=0$, a contradiction. Thus, $\left[1-k_{1}u_{1}\right]=\left[k_{2}\right]$,
and consequently $1=k_{1}u_{1}+k_{2}u_{2}$ for some unit $u_{2}\in K$. This
shows that $K=K_{1}+K_{2}$. Since we already know that $K_{1}\cap K_{2}=0$,
we conclude that $K=K_{1}\times K_{2}$.

Observe that if $x\in K_{1}$ and $x\sim y$ in $K$, then $y\in K_{1}$
and $x\sim y$ in $K_{1}$. Hence, $\varTheta\left(K_{1}\right)\cong G_{1}\cong\varTheta\left(L_{1}\right)$
and similarly $\varTheta\left(K_{2}\right)\cong\varTheta\left(L_{2}\right)$.
\end{proof}

\section{The graph of the ring of integers modulo $m$}

In this section we describe the graph $\varTheta(\mathbb{Z}_m)$ since it will play an important role in the rest of the paper. We remark that graph $\Gamma_E(\mathbb{Z}_m)$ (see \S\ref{intro} for definition) is obtained from $\varTheta(\mathbb{Z}_m)$ by removing vertices $[0]$ and $[1]$ and all loops.

\begin{proposition}\label{lem:staircase--properties}
Let $k$ be a nonnegative integer. Up to graph isomorphism there exists a unique graph $SG_k$ such that $|V(SG_k)|=k+1$ and the degrees of vertices of $SG_{k}$ are $1,2,\ldots,k+1$.
In addition, if we let $v_i\in V\left(SG_{k}\right)$, $0 \leq i \leq k$, be the vertex with degree $i+1$, then $SG_k$ has the following properties:
\begin{enumerate}
\item\label{enu:staircase--neighbourhood} $N\left(v_{i}\right)=\left\{ v_{k-i},v_{k-i+1},\ldots,v_{k}\right\}$ for all $0 \leq i \leq k$,
\item\label{enu:staircase--chain--neighbourhoods} $N\left(v_{0}\right)\varsubsetneq\ldots\varsubsetneq N\left(v_{k-1}\right)\varsubsetneq N\left(v_{k}\right)$.\end{enumerate}
We will call $SG_k$ the \emph{staircase graph} with index $k$.
\end{proposition}

\begin{proof} Let $G$ be a graph with vertices $\{v_0,v_1,v_2,\ldots,v_k\}$, where vertices $v_i$ and $v_j$ (not necessarily distinct) are adjacent if and only if $i+j \geq n$. Then, clearly, $\deg (v_i)=i+1$ for all $0\leq i\leq k$, and graph $G$ satisfies \ref{enu:staircase--neighbourhood} and \ref{enu:staircase--chain--neighbourhoods}. Thus, it remains to prove the uniqueness of $SG_k$.

Let $H$ be any graph with $k+1$ vertices with degrees $1,2,\ldots,k+1$. One of the vertices has to have degree
$k+1$, so it has to be adjacent to every vertex, including itself.
We label that vertex by $u_{k}$.
One of the remaining vertices has to have degree $1$, so it has
no other neighbour besides $u_{k}$. We label that vertex by $u_{0}$.
One of the remaining, not yet labeled, vertices has to have degree $k$, so it has
to be adjacent to every vertex (including itself) except $u_{0}$.
We label that vertex by $u_{k-1}$, and continue.
One of the remaining vertices has to have degree $2$, so it has
no other neighbours besides $u_{k}$ and $u_{k-1}$. We label it by $u_{1}$.
One of the remaining vertices has to have degree $k-1$, so it has
to be adjacent to every vertex (including itself) except $u_{0}$
and $u_{1}$. We label that vertex by $u_{k-2}$.
Continuing this process, we eventually label all the vertices of $H$, and since $H$ has precisely $k+1$ vertices, the labels we use are precisely $u_0,u_1,\ldots,u_k$.
It is clear from the labeling process that we have
\[N(u_i)=\{u_{k-i},u_{k-i+1},\ldots,u_k\},\]
hence the map $H \to G$, defined by $u_i \mapsto v_i$, is a graph isomorphism. This shows the uniqueness of $SG_k$.
\end{proof}

Observe that the adjacency matrix of a staircase graph, with vertices ordered by degree, resembles a staircase, hence the name.

We now describe the zero-divisor graphs of rings $\mathbb{Z}_m$. By slight abuse of notation we will denote the elements of $\mathbb{Z}_m$ simply by integers instead of cosets of integers.

\begin{proposition}\label{prop:only--exponents--matter}
Let $m=p_{1}^{k_{1}}p_{2}^{k_{2}}\cdots p_{n}^{k_{n}}$
be a canonical representation of a positive integer $m$. Then $\varTheta\left(\mathbb{Z}_{m}\right)\cong\prod_{i=1}^{n}SG_{k_{i}}$.
\end{proposition}

\begin{proof}
Observe that $\mathbb{Z}_{m}\cong\prod_{i=1}^{n}\mathbb{Z}_{p_{i}^{k_{i}}}$. It is easy to see that, for a prime $p$ and a nonnegative integer $k$, the graph $\varTheta\left(\mathbb{Z}_{p^k}\right)$ has vertices $[p^0], [p^1], [p^2],\ldots,[p^k]$, and the degree of $[p^j]$ is $j+1$.
Hence, $\varTheta\left(\mathbb{Z}_{p^k}\right)\cong SG_{k}$ by Proposition~\ref{lem:staircase--properties}. The result now follows from Proposition~\ref{prop:preserves--products}.
\end{proof}

Let $G=\varTheta\left(\mathbb{Z}_{m}\right)$ for some positive integer $m$.
By \cite[Lemma~4.3]{Lj}, every vertex in $\varTheta(\mathbb{Z}_m)$ is represented by a uniquely determined positive divisor of $m$. Let $n$ denote the number of distinct prime divisors of $m$. We
will show that the structure of $G$ determines uniquely the number
$n$ and the set of exponents in the canonical representation of $m$.
Starting from graph $G$, with no labels on vertices, we describe how to reconstruct the labels of $G$ (as associatedness classes) in terms of graph properties.
Of course, by Proposition~\ref{prop:only--exponents--matter}, the structure of $G$ does not determine the prime factors of $m$, so besides the graph itself, we will also need additional information on which primes are involved.

The only vertex in $G$ adjacent to every vertex (including itself), is $\left[0\right]$.
The only vertex of degree $1$, is $\left[1\right]$ and it is adjacent
only to $\left[0\right]$.
So we can label these two vertices immediately.

Observe that a vertex $v\in V\left(G\right)$ corresponds to some
prime $p$ if and only if $\deg (v)=2$, since its
neighbours in this case are precisely $\left[0\right]$ and $\left[m/p\right]$.
Since the structure of $G$ does
not determine the prime factors of $m$, we have to assign the degree
$2$ vertices some specific distinct primes, say $p_1,p_2,\ldots,p_n$, where $n$ is just the number of degree 2 vertices in $G$. So, now we have labels
$\left[0\right],\left[1\right],\left[p_{1}\right],\ldots,\left[p_{n}\right]$ and we want to reconstruct the labels of all the other vertices.

For a divisor $d$ of $m$ we will call vertex $\left[m/d\right]$ the complement of vertex $\left[d\right]$.
First we can identify the complements of $\left[p_{1}\right],\ldots,\left[p_{n}\right]$,
since the complement of $\left[p_{i}\right]$, is the unique neighbour
of $\left[p_{i}\right]$ different from $\left[0\right]$. We label the complement of $\left[p_{i}\right]$ by $\left[m/p_{i}\right]$, however this is not a true label yet, since we do not know yet what $m$ is or rather what the exponent of each $p_i$ in the factorization of $m$ is. We determine these exponents now.

Fix some $i\in\left\{ 1,2,\ldots,n\right\} $.
The classes of powers of $p_{i}$ are those neighbours of $\left[m/p_{i}\right]$
that are not neighbours of any $\left[m/p_{j}\right]$, $j\neq i$.
The number of such neighbours gives us the highest power of $p_{i}$
that divides $m$, say $p_{i}^{m_{i}}$. This, in particular, determines $m$ so we can now truly label the complement of each $[p_i]$. We can also label the vertices that correspond to powers of $p_i$. By the above let $v$ be a neighbour
of $\left[m/p_{i}\right]$ that is not a neighbour of any $\left[m/p_{j}\right]$, $j\neq i$. Then the label for $v$ is $\left[p^{\deg (v)-1}\right]$.
This is because the neighbours of $\left[p^{k}\right]$ are precisely
$\left[m\right],\left[m/p_{1}\right],\ldots,\left[m/p_{1}^{k}\right]$.
Observe that we have not changed the label of $\left[p_{i}\right]$ in this step.

Next, we label the complements of powers of primes. Fix some $i\in\left\{ 1,2,\ldots,n\right\} $.
The vertex $\left[m/p_{i}\right]$ is already labeled. For $k \geq 2$, the only neighbour
of $\left[p_{i}^{k}\right]$ that is not a neighbour of $\left[p_{i}^{k-1}\right]$,
must be labeled $\left[m/p_{i}^{k}\right]$.

Finally, we can label all the remaining vertices. Let $v$ be a vertex and for each $i \in \{1,2,\ldots,n\}$ let $k_{i}\leq m_{i}$ be the greatest nonnegative integer such that $v$
is a neighbour of $\left[m/p_{i}^{k_{i}}\right]$. Then the label of $v$ is $\left[p_{1}^{k_{1}}p_{2}^{k_{2}}\ldots p_{n}^{k_{n}}\right]$.

We remark that, although in this algorithm we label some vertices more than once,
the labels are consistent.

Having a labeled zero-divisor graph of $\mathbb{Z}_m$ it is now easy to reconstruct $\Gamma(\mathbb{Z}_m)$, the standard non-compressed zero-divisor graph $\Gamma(\mathbb{Z}_m)$ of $\mathbb{Z}_m$, as defined in \cite{And-Liv}. To do this we first exclude vertices
$\left[0\right]$ and $\left[1\right]$ from $\varTheta(\mathbb{Z}_m)$. Then we replace
each remaining vertex $\left[d\right]$, $d|m$, of $\varTheta(\mathbb{Z}_m)$ by the set 
\[
A_d=\left\{ ds : s\in\left\{ 1,2,\ldots,\frac{m}{d}-1\right\} ,\textrm{gcd}\left(s,m\right)=1\right\}. 
\]
The union of all these sets is the set of vertices of $\Gamma(\mathbb{Z}_m)$. If $\left[d_{1}\right]$ was adjacent to $\left[d_{2}\right]$
in graph $\varTheta(\mathbb{Z}_m)$, then every $x \in A_{d_1}$ is adjacent to every $y \in A_{d_2}$ in $\Gamma(\mathbb{Z}_m)$.
In particular, if there was a loop on vertex $\left[d\right]$ in $\varTheta(\mathbb{Z}_m)$, then $A_{d}$ is a clique (with no loops) in $\Gamma(\mathbb{Z}_m)$, and if there was no loop on $[d]$ in $\varTheta(\mathbb{Z}_m)$, then $A_{d}$ is an independant set in $\Gamma(\mathbb{Z}_m)$.
This ``blow up'' process has already been described by Spiroff and Wickham \cite[end of \S 1]{Spi-Wic}. However, in our situation, conveniently, the loops in $\varTheta(\mathbb{Z}_m)$ determine the edges between the vertices of $A_d$ in $\Gamma(\mathbb{Z}_m)$.
So this blow up process could be done entirely graph-theoretically if one was to encode in $\varTheta(\mathbb{Z}_m)$ also the size of associatedness classes, that is the size of sets $A_d$, say as weights of vertices.

\section{Local rings and principal ideal rings}

Recall that every finite commutative unital ring is isomorphic to a finite direct product of finite local rings (see for example \cite{Ati-Mac}). Hence, a finite commutative unital ring is local if and only if it is directly indecomposable.

\begin{corollary}\label{cor:local--local}
If $\varTheta\left(K\right) \cong \varTheta\left(L\right)$
and $K$ is local, then $L$ is local as well.
\end{corollary}

\begin{proof}
Suppose $L$ is not local. Then $L=L_{1}\times L_{2}$ where $L_{1},L_{2} \neq 0$.
Hence, $\varTheta\left(K\right)=\varTheta\left(L\right)=\varTheta\left(L_{1}\right)\times\varTheta\left(L_{2}\right)$
by Proposition~\ref{prop:preserves--products}.

By Theorem~\ref{prop:preproduct}, there exist subrings $K_{1},K_{2}\subseteq K$
such that $K=K_{1}\times K_{2}$ and $\varTheta\left(K_{1}\right)\cong\varTheta\left(L_{1}\right)$
and $\varTheta\left(K_{2}\right)\cong\varTheta\left(L_{2}\right)$.
Since $L_{1} \neq 0$, we have $\varTheta\left(K_{1}\right)\cong\varTheta\left(L_{1}\right)\ncong\varTheta\left(0\right)$,
hence also $K_{1}\neq 0$. Similarly, $K_2 \neq 0$. This is a contradiction because
$K$ is local.
\end{proof}

From Corollary~\ref{cor:local--local} and the fact that $\varTheta\left(\mathbb{Z}_{p^{n}}\right)\cong SG_{n}$,
we immediately obtain the following result.

\begin{corollary}\label{cor:staircase--local}
If $\varTheta\left(K\right)$ is isomorphic
to the staircase graph $SG_{n}$, then $K$ is a local ring.
\end{corollary}

It turns out that locality of a finite commutative unital ring is a property that can be characterized by the structure of its zero-divisor graph as is shown by the next theorem.
For $a\in K$, we will adopt the convention that $a^{0}=1$ even when $a=0$. If $K$ is a finite local unital ring with maximal ideal $\mathfrak{m}$ then every non-unit of $K$ is contained in $\mathfrak{m}$ and $\mathfrak{m}$ is a nilpotent ideal. Hence, every element of $K$ is either a unit or a nilpotent element.

\begin{theorem}\label{prop:local--iff}
Let $K$ be a finite commutative unital ring.
Then $K$ is local if and only if for all $a,b\in K$ with $\left\{ \left[a\right],\left[b\right]\right\} \cap\left\{ \left[0\right],\left[1\right]\right\} =\emptyset$
we have
\[
N\left(\left[a\right]\right)\cup N\left(\left[b\right]\right)\varsubsetneq N\left(\left[ab\right]\right)
\]
within $\varTheta(K)$.
\end{theorem}

\begin{proof}
Suppose $K$ is local. Then the condition $\left\{ \left[a\right],\left[b\right]\right\} \cap\left\{ \left[0\right],\left[1\right]\right\} =\emptyset$
implies that $a$ and $b$ are nontrivial nilpotents. Let $a^{n}=0$, $b^{m}=0$, $a^{n-1}\neq0$, and $b^{m-1}\neq0$,
where $m,n\geq2$. Suppose $N\left(\left[a\right]\right)\cup N\left(\left[b\right]\right)=N\left(\left[ab\right]\right)$.

We show by induction on $k+l$ that $a^{k}b^{l}\neq 0$ for all $k\in\left\{ 0,1,\ldots,n-1\right\} $
and $l\in\left\{ 0,1,\ldots,m-1\right\} $.
If $k=0$ or $l=0$, this holds by definition of $n$ and $m$. So,
suppose $k,l\geq1$ and assume, on the contrary, that $a^{k}b^{l}=0$.
Then $\left[a^{k-1}b^{l-1}\right]\in N\left(\left[ab\right]\right)=N\left(\left[a\right]\right)\cup N\left(\left[b\right]\right)$.
Hence, either $a^{k}b^{l-1}=0$ or $a^{k-1}b^{l}=0$. But this is
impossible by induction. 

Next, we show by induction on $k+l$ that $N\left(\left[a^{k}b^{l}\right]\right)=N\left(\left[a^{k}\right]\right)\cup N\left(\left[b^{l}\right]\right)$
for all $k,l\geq0$. If $k=0$ or $l=0$, this
is obvious. So, assume $k,l\geq1$. Let $\left[x\right]\in N\left(\left[a^{k}b^{l}\right]\right)$.
Then $\left[a^{k-1}b^{l-1}x\right]\in N\left(\left[ab\right]\right)=N\left(\left[a\right]\right)\cup N\left(\left[b\right]\right)$.
Hence, either $a^{k}b^{l-1}x=0$ or $a^{k-1}b^{l}x=0$. By induction,
the first equality implies $\left[x\right]\in N\left(\left[a^{k}b^{l-1}\right]\right)=N\left(\left[a^{k}\right]\right)\cup N\left(\left[b^{l-1}\right]\right)\subseteq N\left(\left[a^{k}\right]\right)\cup N\left(\left[b^{l}\right]\right)$.
Similiarly, the second equality also implies $\left[x\right]\in N\left(\left[a^{k}\right]\right)\cup N\left(\left[b^{l}\right]\right)$.
This shows that $N\left(\left[a^{k}b^{l}\right]\right)\subseteq N\left(\left[a^{k}\right]\right)\cup N\left(\left[b^{l}\right]\right)$,
hence $N\left(\left[a^{k}b^{l}\right]\right)=N\left(\left[a^{k}\right]\right)\cup N\left(\left[b^{l}\right]\right)$.

Now let $A=a^{n-1}$ and $B=b^{m-1}$. Then, by the above, we have
$A\neq0$, $B\neq0$, $AB\neq0$, $A^{2}=0$, $B^{2}=0$ and $N\left(\left[AB\right]\right)=N\left(\left[A\right]\right)\cup N\left(\left[B\right]\right)$.
This implies $AB\left(A+B\right)=A^{2}B+AB^{2}=0$, so $A+B\in N\left(\left[AB\right]\right)=N\left(\left[A\right]\right)\cup N\left(\left[B\right]\right)$.
Hence, either $0=\left(A+B\right)A=A^{2}+AB=AB$ or $0=\left(A+B\right)B=AB+B^{2}=AB$.
This is a contradiction which shows that $N\left(\left[a\right]\right)\cup N\left(\left[b\right]\right)\varsubsetneq N\left(\left[ab\right]\right)$.

Now, suppose $K$ is not local. Any finite commutative unital ring
is a direct product of local rings, hence $K=K_{1}\times K_{2}$ for
some nonzero rings $K_{1}$ and $K_{2}$. If we take $a=\left(1,0\right)\in K$,
then clearly $\left[a\right]\notin\left\{ \left[0\right],\left[1\right]\right\} $
and $N\left(\left[aa\right]\right)=N\left(\left[a\right]\right)=N\left(\left[a\right]\right)\cup N\left(\left[a\right]\right)$.\end{proof}

\begin{corollary}
Let $K$ be a finite commutative unital local ring which is not a
field. If $\left[a\right]\neq\left[1\right]$ has the least degree
in $\varTheta\left(K\right)$, apart from $\left[1\right]$, then
$a\in K$ is an irreducible element.
\end{corollary}

\begin{proof}
Suppose $a=bc$, where $b$ and $c$ are not units. Observe that $a\neq0$
since $K$ is not a field. Then by Theorem~\ref{prop:local--iff},
$N\left(\left[b\right]\right)\cup N\left(\left[c\right]\right)\varsubsetneq N\left(\left[bc\right]\right)=N\left(\left[a\right]\right)$
which implies $\left|N\left(\left[b\right]\right)\right|,\left|N\left(\left[c\right]\right)\right|<\left|N\left(\left[a\right]\right)\right|$,
a contradiction.
\end{proof}

Suppose $\varTheta\left(K\right)\cong SG_{n}$. Each vertex in this graph corresponds to the associatedness class of some element in $K$. We want to find a nice set of elements that represent the vertices of $\varTheta\left(K\right)$. To this end we need the following two lemmas.

\begin{lemma}\label{prop:local--distributive}
Let $K$ be a local ring with maximal
ideal $\mathfrak{m}$, such that $\varTheta\left(K\right)\cong SG_{n}$.
Denote representatives of associatedness classes in such a way that $N\left(\left[a_{0}\right]\right)\varsubsetneq\ldots\varsubsetneq N\left(\left[a_{n-1}\right]\right)\varsubsetneq N\left(\left[a_{n}\right]\right)$.
Let $i<j$ and $y\in\mathfrak{m}$. Then
\begin{enumerate}
\item\label{enu:local1} $N\left(\left[ya_{i}\right]\right)\subseteq N\left(\left[ya_{j}\right]\right)$, and
\item\label{enu:local2} if $N\left(\left[ya_{i}\right]\right)=N\left(\left[ya_{j}\right]\right)$
then $ya_{i}=ya_{j}=0$.
\end{enumerate}
\end{lemma}

\begin{proof}
\ref{enu:local1} Since $N\left(\left[a_{i}\right]\right)\varsubsetneq N\left(\left[a_{j}\right]\right)$,
we have $\ann\left( a_{i} \right)\varsubsetneq\ann\left( a_{j} \right)$,
hence $\ann\left( ya_{i} \right)\subseteq\ann\left( ya_{j} \right)$,
so $N\left(\left[ya_{i}\right]\right)\subseteq N\left(\left[ya_{j}\right]\right)$.

\smallskip\noindent\ref{enu:local2} Since neighbourhoods of distinct vertices of $\varTheta\left(K\right)\cong SG_{n}$
are distinct, we must have $\left[ya_{i}\right]=\left[ya_{j}\right]$,
hence $ya_{i}=ya_{j}u$ for some unit $u$. So, 
\begin{equation}
y\left(a_{i}-a_{j}u\right)=0.\label{eq:1-3}
\end{equation}
Since $N\left(\left[a_{i}\right]\right)\varsubsetneq N\left(\left[a_{j}\right]\right)$,
we can choose the greatest $k$ such that $\left[a_{k}\right]\in N\left(\left[a_{j}\right]\right)\setminus N\left(\left[a_{i}\right]\right)$.
Then 
\begin{equation}
a_{k}\left(a_{i}-a_{j}u\right)=a_{k}a_{i}\neq0.\label{eq:2-2}
\end{equation}
Since $y\in\mathfrak{m}$, there exists $l>0$ such that $\left[y\right]=\left[a_{l}\right]$.
From (\ref{eq:1-3}) and (\ref{eq:2-2}) we conclude $N\left(\left[a_{l}\right]\right)\nsubseteq N\left(\left[a_{k}\right]\right)$,
hence $l>k$.
Proposition~\ref{lem:staircase--properties} tells us that $N\left(\left[a_{j}\right]\right)=\left\{ \left[a_{n-j}\right],\left[a_{n-j+1}\right],\ldots,\left[a_{n}\right]\right\} $.
This implies $\left[a_{l}\right]\in N\left(\left[a_{j}\right]\right)$,
hence $\left[a_{l}\right]\in N\left(\left[a_{i}\right]\right)$ by
the choice of $k$. Therefore, $a_{l}a_{j}=a_{l}a_{i}=0$ and consequently $ya_{j}=ya_{i}=0$.\end{proof}

\begin{lemma}\label{lem:local--different--powers}
Let $K$ be a local ring with
maximal ideal $\mathfrak{m}$, such that $\varTheta\left(K\right)\cong SG_{n}$.
If $x\in\mathfrak{m}$, with $x^{m}=0$ and $x^{m-1}\neq 0$, then $\left[x^{k}\right]\neq\left[x^{l}\right]$
for all $0\leq k<l\leq m$.
\end{lemma}

\begin{proof}
Suppose otherwise, that $\left[x^{k}\right]=\left[x^{l}\right]$.
Then $x^{k}=x^{l}u$, where $u$ is a unit, so that $x^{k}\left(1-x^{l-k}u\right)=0$.
Since $l-k\geq1$ and $x$ is nilpotent, $1-x^{l-k}u$ is a unit.
But then $x^{k}=0$, a contradiction.
\end{proof}

We can now shows that the vertices of a zero-divisor graph which is isomorphic to a staircase graph can be labeled by powers of a single element of the ring.

\begin{proposition}\label{prop:staircase--labeling}
Let $K$ be a local ring with maximal
ideal $\mathfrak{m}$, such that $\varTheta\left(K\right)\cong SG_{n}$.
Denote representatives of associatedness classes in such a way that $N\left(\left[a_{0}\right]\right)\varsubsetneq\ldots\varsubsetneq N\left(\left[a_{n-1}\right]\right)\varsubsetneq N\left(\left[a_{n}\right]\right)$.
Then $\left[a_{i}\right]=\left[a_{1}^{i}\right]$ for all $i\in\left\{ 0,\ldots,n\right\} $.
\end{proposition}

\begin{proof}
We prove the claim by induction on $i$. Clearly, $\left[a_{0}\right]=\left[1\right]$, so the claim is true for $i=0$ and also for $i=1$.

Let $i\in\left\{ 2,\ldots,n\right\} $. We examine the products
$a_{i-1}a_{1}$, $a_{i-1}a_{2}$, $a_{i-1}a_{3}$, $\ldots$, $a_{i-1}a_{n-i}$, $a_{i-1}a_{n-i+1}$.
By Proposition~\ref{lem:staircase--properties} we have
\[N\left(\left[a_{i-1}\right]\right)=\left\{ \left[a_{n-i+1}\right],\left[a_{n-i+2}\right],\ldots,\left[a_{n}\right]\right\},\]
therefore $a_{i-1}a_{1}\neq0$, $a_{i-1}a_{2}\neq0,\ldots$, $a_{i-1}a_{n-i}\neq0$ and $a_{i-1}a_{n-i+1}=0$.
Hence, by Lemma~\ref{prop:local--distributive}, 
\[
N\left(\left[a_{i-1}a_{1}\right]\right)\varsubsetneq N\left(\left[a_{i-1}a_{2}\right]\right)\varsubsetneq\ldots\varsubsetneq N\left(\left[a_{i-1}a_{n-i}\right]\right)\varsubsetneq N\left(\left[a_{i-1}a_{n-i+1}\right]\right).
\]
Since this is a subchain of the chain $N\left(\left[a_{0}\right]\right)\varsubsetneq\ldots\varsubsetneq N\left(\left[a_{n-1}\right]\right)\varsubsetneq N\left(\left[a_{n}\right]\right)$,
we conclude that $N\left(\left[a_{i-1}a_{1}\right]\right)\subseteq N\left(\left[a_{i}\right]\right)$.
By induction, $[a_{i-1}]=[a_1^{i-1}]$, which implies $[a_{i-1}a_1]=[a_1^{i-1}a_1]=[a_1^i]$. Thus, $N\left(\left[a_{1}^{i}\right]\right)\subseteq N\left(\left[a_{i}\right]\right)$,
so there exists $j\leq i$ such that $\left[a_{1}^{i}\right]=\left[a_{j}\right]$.
If $j<i$, then by induction $\left[a_{1}^{i}\right]=[a_{1}^{j}]$,
which contradicts Lemma~\ref{lem:local--different--powers} unless
$a_1^{i-1}=0$. But the latter would imply $\left[a_{i-1}\right]=\left[a_{1}^{i-1}\right]=\left[0\right]=\left[a_{n}\right]$ and consequently $i=n+1$ which is not the case.
\end{proof}

Recall that a \emph{principal ideal ring}, abbreviated PIR, is a commutative unital ring in which every ideal is principal. Being a PIR is another property that can be characterized by the structure of $\varTheta(K)$.

\begin{theorem}\label{thm:pir-iff}
A finite commutative unital ring $K$ is a PIR if and only if $\varTheta\left(K\right)$ is isomorphic
to a finite tensor product of staircase graphs.
\end{theorem}

\begin{proof}
Suppose $\varTheta\left(K\right)\cong\prod_{i=1}^{n}SG_{k_{i}}\cong SG_{k_{1}}\times\prod_{i=2}^{n}SG_{k_{i}}$.
By Proposition~\ref{prop:only--exponents--matter}, we have $\varTheta\left(K\right)\cong\varTheta\left(\mathbb{Z}_{p_{1}^{k_{1}}}\right)\times\varTheta\left(\mathbb{Z}_{p_{2}^{k_{2}}\cdots p_{n}^{k_{n}}}\right)$
for some distinct primes $p_{1},p_{2},\ldots,p_{n}$. Theorem
\ref{prop:preproduct} implies $K\cong K_{1}\times K_{1}'$, where
$\varTheta\left(K_{1}\right)\cong SG_{k_{1}}$ and $\varTheta\left(K_{1}'\right)\cong\prod_{i=2}^{n}SG_{k_{i}}$.
By induction, $K\cong\prod_{i=1}^{n}K_{i}$, where $\varTheta\left(K_{i}\right)\cong SG_{k_{i}}$.
Since the direct product of PIR's is a PIR, it suffices to prove that
each $K_{i}$ is a PIR. By Corollary~\ref{cor:staircase--local},
$K_{i}$ is a local ring. Denote its maximal ideal by $\mathfrak{m}_{i}$.
Then by Proposition~\ref{prop:staircase--labeling}, there exists $x\in K_{i}$
such that $\left[x^{0}\right],\left[x^{1}\right],\left[x^{2}\right],\ldots,\left[x^{k_{i}}\right]$
are all of the vertices of $\varTheta\left(K_{i}\right)$. This clearly
implies that every ideal of $K_{i}$ is principal, generated by the
least power of $x$ it contains.

Conversely, suppose $K$ is a PIR. Then by a result of Hungerford \cite[Theorem~1]{Hun}, $K$ is a finite direct product of homomorphic images
of PID's, say $K\cong\prod_{i=1}^{n}K_{i}/I_{i}$, where $K_{i}$
is a PID (not necessarily finite) and $I_{i}\triangleleft K_{i}$
for all $i\in\left\{ 1,2,\ldots,n\right\} $. By Proposition~\ref{prop:preserves--products},
it suffices to prove that each $\varTheta\left(K_{i}/I_{i}\right)$
is a tensor product of staircase graphs. If $I_{i}=0$, then $K_{i}$
has to be finite and every finite PID is a field. In this case, $\varTheta\left(K_{i}/I_{i}\right)$
is isomorphic to either $SG_{1}$ or $SG_{0}$. If $I_{i}=K_{i}$,
then $\varTheta\left(K_{i}/I_{i}\right)$$\cong SG_{0}$. Now, assume
$0\neq I_{i}\neq K_{i}$. Then, $I_{i}$ is generated by some $a_{i}=u\cdot p_{1}^{\alpha_{1}}p_{2}^{\alpha_{2}}\cdots p_{m}^{\alpha_{m}}$,
where $m\geq1$, $\alpha_{1}\geq1$, $p_{j}$ are prime elements and $u$
is a unit in $K_{i}$. By the Chinese Remainder Theorem, $K_{i}/I_{i}\cong\prod_{j=1}^{m}K_{i}/(p_{j}^{\alpha_{j}})$,
hence it suffices to prove that $\varTheta\left(K_{i}/(p_{j}^{\alpha_{j}})\right)\cong SG_{\alpha_{j}}.$
This is easily shown upon observing that every element in $K_{i}/(p_{j}^{\alpha_{j}})$
is a product of some power of $p_{j}$ and some unit.
\end{proof}

The following corollary easily follows from the proof of Theorem~\ref{thm:pir-iff}.

\begin{corollary}\label{cor:local_PIR}
Let $K$ be a finite commutative unital ring. Then $K$ is a local PIR if and only if $\varTheta\left(K\right)\cong SG_{n}$ for some nonnegative integer $n$. In fact, $n$ is the index of nilpotency of the maximal ideal of $K$.
\end{corollary}

We remark that for a fixed positive integer $n$ there exist many non-isomorphic local PIR's with $\varTheta\left(K\right)\cong SG_{n}$.
For example, the rings $\mathbb{Z}_{16}$, $\mathbb{Z}_2[x]/(x^4)$, $\mathbb{Z}_4[x]/(x^2-2)$ and $\mathbb{Z}_4[x]/(x^2-2x-2)$ all have the compressed zero-divisor graph $\varTheta\left(K\right)$ isomorphic to $SG_3$ and, in addition, they all have the residue field isomorphic to $\mathbb{Z}_2$.
Moreover, they have the corresponding associatedness classes of the same sizes, which means that they also have the same non-compressed zero-divisor graphs $\Gamma(K)$. All the above can be verified by hand and we leave the verification to the reader.

Corollary~\ref{cor:local_PIR} shows that for a finite local PIR $K$ the index of nilpotency of its maximal ideal can be extracted from the structure of $\varTheta(K)$. We were not able to establish whether the same holds for any finite local ring so we leave it as an open question.

\begin{question}
Let $K$ be a finite local unital ring with maximal ideal $\mathfrak{m}$.
\begin{enumerate}
\item[(a)] Does the graph structure of $\varTheta(K)$ determine the index of nilpotency of $\mathfrak{m}$?
\item[(b)] Does the graph structure of $\varTheta(K)$ determine the minimal number of generators of $\mathfrak{m}$?
\end{enumerate}
\end{question}

\section{Infinite rings}

Finally, we remark that the definition of graph $\varTheta(K)$ can be extended to infinite commutative unital rings, however a verbatim extension is not the best way to do so.
In view of the proof of Proposition~\ref{prop:best}, we believe that the right way to extend the definition is to compress the zero-divisor graph by the relation $\approx$, defined by $a \approx b$ if and only if $aK=bK$, and define edges in a similar way as in the finite case.
By this definition, the equivalence classes are in a bijective correspondence $[a]_\approx \leftrightarrow aK$ with the principal ideals of $K$, and two classes are connected by an edge if and only if the product of the corresponding principal ideals is $0$. Hence, we propose the following extension of Definition~\ref{def:theta}.

\begin{definition}\label{def:extension}
For an arbitrary commutative unital ring $K$, $\varTheta\left(K\right)$
is a graph whose vertices are principal ideals of $K$ (including $0$ and $K$) and vertices $I$ and $J$ (not necessarily
distinct) are adjacent if and only if $IJ=0$.
\end{definition}

As remarked by Kaplansky in \cite[\S 2]{Kap}, for any artinian commutative unital ring $K$, the equality $aK=bK$ holds if and only if $a \sim b$. Hence, for any artinian commutative unital ring, and in particular for any finite commutative unital ring, Definition~\ref{def:extension} is equivalent to Definition~\ref{def:theta}.

\end{document}